\documentclass[12pt]{amsart}

\usepackage[english]{babel}
\usepackage{mathrsfs,amssymb}
\usepackage{mathtools}
\usepackage[colorlinks, citecolor = blue]{hyperref}

\usepackage[shortlabels]{enumitem}
\setlist[itemize]{leftmargin=25pt}
\setlist[enumerate]{leftmargin=25pt}

\newtheorem{theorem}{Theorem}[section]
\newtheorem{lemma}[theorem]{Lemma}
\newtheorem{prop}[theorem]{Proposition}
\newtheorem{cor}[theorem]{Corollary}

\theoremstyle{definition}
\newtheorem{definition}[theorem]{Definition}
\newtheorem{que}[theorem]{Question}

\theoremstyle{remark}
\newtheorem{remark}[theorem]{Remark}

\numberwithin{equation}{section}
\usepackage[colorinlistoftodos,prependcaption,textsize=small]{todonotes}

\let \la=\lambda
\let \e=\varepsilon
\let \d=\delta

\let \a=\alpha

\allowdisplaybreaks

\begin{document}
\title[characterization of the $A_p$ condition]
{On a non-standard characterization of the $A_p$ condition}

\author[A.K. Lerner]{Andrei K. Lerner}
\address[A.K. Lerner]{Department of Mathematics,
Bar-Ilan University, 5290002 Ramat Gan, Israel}
\email{lernera@math.biu.ac.il}

\thanks{The author was supported by ISF grant no. 1035/21.}

\begin{abstract}
The classical Muckenhoupt's $A_p$ condition is necessary and sufficient for the boundedness of the maximal operator $M$ on
$L^p(w)$ spaces. In this paper we obtain another characterization of the $A_p$ condition. As a result, we show that some strong
versions of the weighted $L^p(w)$ Coifman--Fefferman and Fefferman--Stein inequalities hold if and only if $w\in A_p$. We also
give new examples of Banach function spaces $X$ such that $M$ is bounded on $X$ but not bounded on the associate space $X'$.
\end{abstract}


\maketitle

\section{Introduction}
Let $M$ be the Hardy--Littlewood maximal operator defined by
$$Mf(x):=\sup_{Q\ni x}\frac{1}{|Q|}\int_Q|f|,$$
where the supremum is taken over all cubes $Q\subset {\mathbb R}^n$ containing the point $x$.
By a weight we mean a non-negative locally integrable function on ${\mathbb R}^n$. A weight $w\in A_p, p>1,$ if
$$[w]_{A_p}:=\sup_Q\Big(\frac{1}{|Q|}\int_Qw\Big)\Big(\frac{1}{|Q|}\int_Qw^{-\frac{1}{p-1}}\Big)^{p-1}<\infty.$$
By the classical Muckenhoupt's theorem \cite{M72}, $M$ is bounded on $L^p(w)$ iff $w\in A_p$.

Denote by $N_p, p>1,$ the class of weights satisfying
$$\int_{{\mathbb R}^n}\frac{w(x)}{(1+|x|)^{np}}dx<\infty.$$
Since $Mf(x)\ge \frac{C}{(1+|x|)^n}$ whenever $f$ is not identically zero, we obtain that $A_p\subset N_p$.
It is obvious that this inclusion is strict (for example, any bounded weight belongs to $N_p$ while there is a variety of
bounded weights not belonging to $A_p$).

Now, denote by $W_p,p>1,$ the class of weights $w\in N_p$ for which there exists a constant $C>0$ such that
$$\int_{{\mathbb R}^n}\big(M(Mf)\big)^pw\le C\int_{{\mathbb R}^n}(Mf)^pw$$
whenever the right-hand side is finite. Then, we trivially have that $A_p\subseteq W_p$.
At first glance, it seems that this inclusion is also strict. However, this is still not clear to us.
We formulate this question as follows.

\begin{que}\label{Q}
Does there exist a weight $w\in W_p\setminus A_p$?
\end{que}

In this paper we obtain a result related to this question, and which is of some independent interest.
In order to state it, we need a definition of the $C_p$ class of weights.

We say that $w\in C_p, p>0,$ if there exist $C,\d>0$ such that for every cube $Q$ and any subset $E\subset Q$,
$$\int_Ew\le C\left(\frac{|E|}{|Q|}\right)^{\d}\int_{{\mathbb R}^n}(M\chi_Q)^pw.$$
This condition was introduced by Muckenhoupt \cite{M80}, and it plays an important role in the study of weighted Coifman--Fefferman
and Fefferman--Stein inequalities (see, e.g., \cite{CLRT21,S83,Y90}).

Denote $A_{\infty}:=\cup_{p>1}A_p$. It is well known that for any $p>1$, $A_p\subsetneq A_{\infty}\subsetneq C_p$, that is, the class $C_p$ is
much wider than $A_p$. In particular, a~$C_p$ weight may vanish on a set of positive measure while a weight from $A_p$ must be positive almost
everywhere (see, e.g., \cite{M80} or \cite{CLRT21}).

Since $A_p\subseteq W_p$, we have that $A_p\subseteq W_p\cap C_p$. Our main result says that in fact
the converse inclusion holds as well and we have the following non-standard characterization of the $A_p$ condition.

\begin{theorem}\label{mr} For any $p>1$,
$$A_p=W_p\cap C_p.$$
\end{theorem}

Returning to Question \ref{Q}, observe the following. As we will show below, in the proof of Theorem \ref{mr}, any weight $w$ in $W_p$ satisfies the doubling condition (which we denote by $w\in D$), namely, there exists a constant $C>0$ such that $\int_{2Q}w\le C\int_Qw$ for every cube $Q$. Therefore, if there exists a weight $w\in W_p\setminus A_p$, such a weight must necessarily belong to $D\setminus C_p$. But even a construction of weights in $D\setminus A_{\infty}$ is rather non-trivial, the first example of such a weight was given by Fefferman and Muckehnoupt~\cite{FM74}. This only says that an attempt to give a positive answer to Question \ref{Q} is
not a simple task.

In what follows, we consider several applications of Theorem~\ref{mr}, which will explain our interest in the class $W_p$.

\subsection{An application to the space $ML^p(w)$}
Let $X$ be a Banach function space (BFS) over ${\mathbb R}^n$, and let $X'$ denote the associate space.
We refer to a recent survey by Lorist and Nieraeth~\cite{LN24}
where, in particular, one can find a discussion about the choice of axioms needed to define the notion of BFS correctly.

Given a BFS $X$, denote by $MX$ the space equipped with norm $$\|f\|_{MX}:=\|Mf\|_{X}.$$
Observe that when $X:=L^p(w)$, the space $ML^p(w)$ was mentioned in Stein's book \cite[p. 222]{S93}.
For general BFS $X$, the space $MX$ was considered in a recent paper \cite{L20}. In particular, it was shown there that if $X$ is a BFS, then $MX$ is also a BFS
iff $\frac{1}{1+|x|^n}\in X$. It was also proved in \cite{L20} that the Fefferman--Stein inequality on $X$,
$$\|Mf\|_{X}\le C\|f^{\#}\|_{X},$$
holds iff $M$ is bounded on $(MX)'$. Here $f^{\#}$ is the Fefferman--Stein sharp function defined by
$$f^{\#}(x):=\sup_{Q\ni x}\frac{1}{|Q|}\int_Q|f-\langle f\rangle_Q|,\quad \langle f\rangle_Q:=\frac{1}{|Q|}\int_Qf.$$

Suppose now that $X:=L^p(w)$. Then $W_p$ is the class of weights for which $M$ is bounded on $ML^p(w)$. Next, it is well known (see \cite{Y90})
that the $C_p$ condition is necessary for the Fefferman--Stein inequality on $L^p(w)$ (or, equivalently, for the boundedness of $M$ on $(ML^p(w))'$)
and $C_{p+\e},\e>0,$ is sufficient. Thus, we obtain the following corollary from Theorem \ref{mr}.

\begin{cor}\label{fscor} Let $p>1$. If the maximal operator $M$ is bounded on $ML^p(w)$ and on $(ML^p(w))'$, then $w\in A_p$.
\end{cor}

It is well known that for many ``standard" BFS $X$, the maximal operator $M$ is bounded on $X$ iff $M$ is bounded on $X'$,
and there are only few known non-trivial examples when this is not true (see a recent work by Nieraeth \cite{N24} for a
thorough discussion about this topic). Corollary~\ref{fscor} provides a variety of new examples.
Indeed, let $X:=(ML^p(w))'$. Take any weight $w\in C_{p+\e}\setminus A_p$. Then, as discussed above, $M$ is bounded on $X$ but, by Corollary \ref{fscor},
$M$ is not bounded on $X'=ML^p(w)$. Next, observe that if the answer to Question \ref{Q} is positive, we would obtain yet more examples. Indeed, taking
$w\in W_p\setminus A_p$, we would obtain that $M$ is bounded on $ML^p(w)$ and not bounded on $(ML^p(w))'$. Note that it was conjectured in \cite{N24} that
if $X$ is an r-convex and s-concave BFS, where $1<r<s<\infty$, then $M:X\to X$ iff $M:X'\to X'$. However, the above examples do not seem to satisfy the
convexity/concavity assumptions, and so they probably cannot be used to disprove this conjecture.

\subsection{On some variants of Coifman--Fefferman and Fefferman--Stein inequalities}
Suppose that $T$ is a singular convolution integral which is non-degenerate in the sense of Stein \cite[p. 210]{S93}. Then
$T$ is bounded on $L^p(w)$ iff $w\in A_p$. Now we have that the same result holds on $ML^p(w)$ as well.

\begin{theorem}\label{scf} Let $T$ be a non-degenerate singular integral operator. Then $T$ is bounded on $ML^p(w), p>1,$ iff $w\in A_p$.
\end{theorem}

Observe that the sufficiency of $w\in A_p$ trivially follows from the classical theory. Indeed, if $w\in A_p$, then
$$\|M(Tf)\|_{L^p(w)}\le C\|f\|_{L^p(w)}\le C\|Mf\|_{L^p(w)}.$$ However, the necessity of $w\in A_p$ is not trivial even for the Hilbert transform.
Here we essentially use a recent result of Nieraeth \cite{N24} saying that if $X$ is an order-continuous BFS and $T$ is a non-degenerate linear operator bounded on $X$,
then $M$ is bounded on $X$ and $X'$ (an earlier version of this result with more restrictive convexity assumptions on $X$ was obtained by Rutsky \cite{R14}).
Thus, basically the necessity part in Theorem \ref{scf} follows by taking $X:=ML^p(w)$ and applying Corollary~\ref{fscor}. There are still some technicalities related to order-continuity of $ML^p(w)$,
which will be discussed in Section 3.

Theorem \ref{scf} can be viewed as a complement to the Coifman--Fefferman inequality \cite{CF74} saying that
$$\|T^*f\|_{L^p(w)}\le C\|Mf\|_{L^p(w)},$$
where $T^*$ stands for the maximal singular integral operator. A sufficient condition for this estimate to hold is $w\in C_{p+\e}, \e>0$ (see \cite{S83}).
Now, it is easy to show (this can be proved exactly as \cite[Lemma~3.2]{L20}) that for every $r\in (0,1)$ and for all $x\in {\mathbb R}^n$,
$$M_{r}(Tf)(x)\le C_{r,n}(T^*f(x)+Mf(x)),$$
where $M_rf:=M(|f|^r)^{1/r}$. Therefore, if $w\in C_{p+\e}$, then
$$\|M_r(Tf)\|_{L^p(w)}\le C_{r,n}\|Mf\|_{L^p(w)},\quad r\in (0,1).$$
It is natural to wonder what happens when $r=1$, and whether the condition $w\in C_{p+\e}$ remains to be sufficient. Theorem \ref{scf} shows that this is not the case,
and that the above inequality for $r=1$ holds iff $w\in A_p$.

As a simple corollary, we obtain that a similar phenomenon occurs with the Fefferman--Stein inequality~\cite{FS72} saying that
\begin{equation}\label{FS}
\|Mf\|_{L^p(w)}\le C\|f^{\#}\|_{L^p(w)}.
\end{equation}
By \cite{Y90}, the $C_{p+\e}$ condition is sufficient for (\ref{FS}). Suppose now that we change $f^{\#}$ in (\ref{FS}) by a smaller operator
$f_{\d}^{\#}:=\big((|f|^{\d})^{\#}\big)^{1/\d}$ for $\d\in (0,1)$. Then using that
$(Tf)_{\d}^{\#}\le C_{\d,n}Mf$ for a singular integral operator $T$ (see, e.g., \cite{AP94}), we obtain that $T$ is bounded on $ML^p(w)$, and hence $w\in A_p$.
Thus, we have the following.

\begin{cor}\label{fs} Suppose that $p>1$ and $\d\in (0,1)$. Then the estimate
\begin{equation}\label{vfs}
\|Mf\|_{L^p(w)}\le C\|f_{\d}^{\#}\|_{L^p(w)}
\end{equation}
holds iff $w\in A_p$.
\end{cor}

Observe that there is a simpler way to prove Corollary \ref{fs}, without the use of Theorem \ref{scf}. Indeed, (\ref{vfs}) implies (\ref{FS}) and hence, by \cite{Y90}, $w\in C_p$.
Next, using that $f_{\d}^{\#}\le 2M_{\d}f$ and that, by the Coifman--Rochberg theorem \cite{CR80}, $M_{\d}(Mf)\le CMf$, we obtain from (\ref{vfs}) that $M$ is bounded on $ML^p(w)$,
and it remains to apply Theorem \ref{mr}.

\vskip 1mm
The paper is organized as follows. In Section 2 we prove Theorem~\ref{mr}, and in Section 3 we prove Theorem \ref{scf}.

\section{Proof of Theorem \ref{mr}}
An important role in our proof will be played by the local maximal operator defined by
$$m_{\la}f(x):=\sup_{Q\ni x}(f\chi_Q)^*(\la|Q|),\quad \la\in (0,1),$$
where the supremum is taken over all cubes $Q\subset {\mathbb R}^n$ containing the point $x$,
and $(f\chi_Q)^*(\la|Q|)$ denotes the non-increasing rearrangement defined by
$$(f\chi_Q)^*(\la|Q|):=\inf\{\a>0:|\{x\in Q:|f(x)|>\a\}|\le \la|Q|\}.$$
We mention several simple propositions which will be used below.

\begin{prop}\label{pr1} For any $r>1$ and $\la\in (0,1)$, and for all $x\in {\mathbb R}^n$,
$$Mf(x)\le \frac{r}{r-1}\la^{\frac{r-1}{r}}M_rf(x)+m_{\la}f(x).$$
\end{prop}

\begin{proof} By Chebyshev's inequality, for $\tau\in (0,1)$ and for every cube $Q$,
\begin{equation}\label{ch}
(f\chi_Q)^*(\tau|Q|)\le \frac{1}{\tau^{1/r}}\Big(\frac{1}{|Q|}\int_Q|f|^r\Big)^{1/r}.
\end{equation}
Therefore, for $x\in Q$,
\begin{eqnarray*}
\frac{1}{|Q|}\int_Q|f|&=&\frac{1}{|Q|}\int_0^{|Q|}(f\chi_Q)^*(t)dt=\int_0^1(f\chi_Q)^*(\tau|Q|)d\tau\\
&\le& \Big(\int_0^{\la}\frac{1}{\tau^{1/r}}d\tau\Big)\Big(\frac{1}{|Q|}\int_Q|f|^r\Big)^{1/r}+\int_{\la}^1(f\chi_Q)^*(\tau|Q|)d\tau\\
&\le& \frac{r}{r-1}\la^{\frac{r-1}{r}}M_rf(x)+m_{\la}f(x),
\end{eqnarray*}
and the result follows.
\end{proof}

\begin{prop}\label{pr2}  For any $\la\in (0,1)$ and for all $x\in {\mathbb R}^n$,
$$m_{\la}(Mf)(x)\le C_{\la,n}Mf(x).$$
\end{prop}

\begin{proof} By (\ref{ch}), for every $\d>0$,
\begin{equation}\label{mla}
m_{\la}f(x)\le \frac{1}{\la^{1/\d}}M_{\d}f(x).
\end{equation}
Next, by the Coifman--Rochberg theorem \cite{CR80}, $M_{\d}(Mf)\le C_{\d,n}Mf$ for $\d\in (0,1)$.
Combining both estimates proves the result.
\end{proof}

The following result is the key ingredient of the proof, and it can be formulated for general BFS $X$.

\begin{theorem}\label{eqst} Let $X$ be a BFS such that $\frac{1}{1+|x|^n}\in X$. The following statements are equivalent:
\begin{enumerate}[(i)]
\item
$M$ is bounded on $MX$, that is, there exists a constant $C>0$ such that for every $f\in MX$,
$$\|MMf\|_{X}\le C\|Mf\|_{X};$$
\item there exists $\la_0\in (0,1)$ such that for every $f\in X$,
\begin{equation}\label{mx}
\|Mf\|_{X}\le 2\|m_{\la_0}f\|_{X}.
\end{equation}
\end{enumerate}
\end{theorem}

\begin{proof}
Let us prove the implication $({\rm i})\Rightarrow ({\rm ii})$. We will use the well known fact (see, e.g., \cite{LO10}) that if
$M$ is bounded on a BFS $X$, then $M_r$ is bounded on $X$ for some $r>1$. Therefore, the boundedness of $M$ on $MX$
implies that there exist $r>1$ and $C>0$ such that
$$\|M_rf\|_{X}\le C\|Mf\|_{X}.$$
Combining this with Proposition \ref{pr1} yields
$$\|Mf\|_X\le C\frac{r}{r-1}\la^{\frac{r-1}{r}}\|Mf\|_{X}+\|m_{\la}f\|_{X}.$$
Therefore, taking $\la$ such that $C\frac{r}{r-1}\la^{\frac{r-1}{r}}=\frac{1}{2}$ proves (\ref{mx}) under the assumption that $f\in MX$.

Take now an arbitrary $f\in X$. For $N>0$ set
$$f_N:=\min(|f|, N)\chi_{\{|x|\le N\}}.$$
Then $f_N\in MX$. Hence, by (\ref{mx}) for $f\in MX$ we obtain
$$\|M(f_N)\|_{X}\le 2\|m_{\la_0}(f_N)\|_{X}\le 2\|m_{\la_0}f\|_{X}.$$
Since $f_N\uparrow |f|$, we have $M(f_N)\uparrow Mf$ (the proof of this simple fact can be found in, e.g., \cite[Section 2]{L20}).
Therefore, letting $N\to \infty$ completes the proof of (\ref{mx}) for any $f\in X$ by the Fatou property of $X$.

Now, the implication $({\rm ii})\Rightarrow ({\rm i})$ follows immediately by setting $Mf$ instead of $f$ in (\ref{mx}) and applying Proposition \ref{pr2}.
This completes the proof.
\end{proof}

Suppose that $X:=L^p(w)$. Theorem \ref{eqst} shows that if $M$ is bounded on $ML^p(w)$, then $M$ is bounded on $L^p(w)$ (and hence $w\in A_p$) iff $m_{\la}$ does.
But the boundedness of $m_{\la}$ on $L^p(w)$ follows easily assuming just $w\in A_{\infty}$. Indeed, we have the following.

\begin{prop}\label{pr3} Let $w\in A_{\infty}$. Then the local maximal operator $m_{\la}$ is bounded on $L^p(w)$ for every $p>0$ and $\la\in (0,1)$.
\end{prop}

\begin{proof} There exists $r>1$ such that $w\in A_r$. Define $\d:=\frac{p}{r}$. Then applying (\ref{mla}) and using that $M$ is bounded on $L^r(w)$, we obtain
$$\|m_{\la}f\|_{L^p(w)}\le \frac{1}{\la^{1/\d}}\|M(|f|^{\d})\|_{L^r(w)}^{r/p}\le C\|f\|_{L^p(w)},$$
and the proof is complete.
\end{proof}

Thus our strategy in the proof of Theorem \ref{mr} is to deduce that $w\in A_{\infty}$. We will use the following equivalent definition of
$A_{\infty}$ (see, e.g., \cite[p. 527]{G14}): $w\in A_{\infty}$ if there exist
$C,\d>0$ such that for every cube $Q$ and an arbitrary subset $E\subset Q$,
$$\int_Ew\le C\Big(\frac{|E|}{|Q|}\Big)^{\d}\int_Qw.$$

\begin{proof}[Proof of Theorem \ref{mr}]
By Theorem \ref{eqst}, there exists $\la_0\in (0,1)$ such that
\begin{equation}\label{im}
\|Mf\|_{L^p(w)}\le 2\|m_{\la_0}f\|_{L^p(w)}.
\end{equation}

Our goal now is to show that $w\in A_{\infty}$. Then (\ref{im}) combined with Proposition \ref{pr3} would imply that $M$ is bounded on $L^p(w)$ and hence $w\in A_p$.

First we show that $w$ satisfies the doubling condition. It follows from (\ref{im}) that for every cube $Q$,
$$\Big(\frac{1}{|Q|}\int_Q|f|\Big)w(Q)^{1/p}\le 2\|m_{\la_0}(f\chi_Q)\|_{L^p(w)}.$$
Take here $f=\chi_{\e Q}$ with $\e>0$ small enough. We obtain
\begin{equation}\label{step}
w(Q)^{1/p}\le \frac{2}{\e^n}\|m_{\la_0}(\chi_{\e Q})\|_{L^p(w)}.
\end{equation}
Now we claim that if $\e^n<\la_0\big(\frac{1}{4}-\frac{\e}{2}\big)^n$, then $m_{\la_0}(\chi_{\e Q})\le \chi_{\frac{1}{2}Q}$.
Indeed, let $x\not\in \frac{1}{2}Q$ and let $R$ be an arbitrary cube such that $x\in R$ and $R\cap \e Q\not=\emptyset$. Then $|R|\ge \big(\frac{1}{4}-\frac{\e}{2}\big)^n|Q|$.
From this,
$$|R\cap \e Q|\le \e^n|Q|<\la_0|R|,$$
which implies
$$m_{\la_0}(\chi_{\e Q})(x)=\sup_{R\ni x}(\chi_{R\cap\e Q})^*(\la_0|R|)=0,\quad x\not\in \frac{1}{2}Q.$$
Thus, we obtain from (\ref{step}) that for suitably chosen $\e$,
$$w(Q)\le \Big(\frac{2}{\e^n}\Big)^pw(Q/2),$$
which proves the doubling condition.

Now, combining the $C_p$ condition with (\ref{im}) yields that there exist $C,\d>0$ such that for every cube $Q$ and any measurable subset $E\subset Q$,
\begin{equation}\label{alm}
w(E)\le C\Big(\frac{|E|}{|Q|}\Big)^{\d}\|M\chi_Q\|_{L^p(w)}^p\le 2^pC\Big(\frac{|E|}{|Q|}\Big)^{\d}\|m_{\la_0}(\chi_Q)\|_{L^p(w)}^p.
\end{equation}
Applying the same argument as above, we obtain that
$$m_{\la_0}(\chi_Q)\le \chi_{rQ}$$
for $r>1$ satisfying $\big(\frac{r-1}{2})^n=\frac{1}{\la_0}$. Therefore, by the doubling property,
$$\|m_{\la_0}(\chi_Q)\|_{L^p(w)}^p\le w(rQ)\le Cw(Q),$$
which along with (\ref{alm}) proves that $w\in A_{\infty}$, and hence the proof is complete.
\end{proof}

\begin{remark}\label{depc} It can be easily checked that the $A_p$ constant $[w]_{A_p}$ in Theorem \ref{mr} depends only on the $ML^p(w)$-norm of $M$ and on the constants from the definition of $C_p$.
Similarly, the $A_p$ constant in Corollary~\ref{fscor} depends only on the $ML^p(w)$ and $(ML^p(w))'$-norms of $M$.
\end{remark}

\section{Proof of Theorem \ref{scf}}
As we will see, Theorem \ref{scf} is a particular case of a more general result, Theorem \ref{msr}, proved below.
Let us start with some preliminary facts that will be needed in the proof of this result. The following definition was given in \cite{N24}.

\begin{definition}\label{nd}
We say that a (sub)linear operator $T$ is non-degenerate if there exists a constant $C>0$ such that
for all $\ell>0$ there is an $x_{\ell}\in {\mathbb R}^n$ such that for all cubes $Q$ with the side length $\ell_Q=\ell$, for all
non-negative and locally integrable $f$ and for all $x\in (Q+x_{\ell})\cup (Q-x_{\ell})$,
$$\frac{1}{|Q|}\int_Qf\le C|T(f\chi_Q)(x)|.$$
\end{definition}

\begin{remark}\label{std} Recall that a convolution singular integral operator $Tf:=f*K$ is non-degenerate in the sense of Stein \cite[p. 210]{S93}
if, additionally to the standard assumptions on $K$, there exists a constant $C>0$ and a unit vector $u_0$, so that
$$|K(tu_0)|\ge \frac{C}{|t|^n}\quad\text{for all}\,\,t\in {\mathbb R}.$$
For example, this condition holds if $T$ is any one of the Riesz transforms.
An argument given in \cite[p. 211]{S93} shows that if $T$ is non-degenerate in the sense of Stein, then it is also non-degenerate in the sense of Definition \ref{nd}.
\end{remark}

We say that a BFS $X$ is order-continuous if for every sequence $\{f_j\}$ in $X$ such that $f_j\downarrow 0$ almost everywhere we have
$\|f_j\|_{X}\downarrow 0$.
The following statement is an abridged version of \cite[Theorem A]{N24}.

\begin{theorem}[\cite{N24}]\label{N}
Let $X$ be an order-continuous BFS, and let $T$ be a non-degenerate (in the sense of Definition \ref{nd}) linear operator. If $T$ is bounded on $X$, then the maximal operator $M$ is bounded on $X$ and on $X'$.
\end{theorem}

\begin{remark}\label{ndep}
It can be easily checked from the proof of Theorem \ref{N} that the $X$ and $X'$-norms of $M$ depend only on the $X$-norm of $T$.
\end{remark}

Now, our goal is to apply Theorem \ref{N} to $X:=ML^p(w)$. Therefore, it is important to check whether $ML^p(w)$ is order-continuous.
The definition of $ML^p(w)$ makes sense if $w\in N_p$ and $w>0$ on a set of positive measure, and so we assume that these two conditions hold.
Then there is a simple characterization of when $ML^p(w)$ is order-continuous.

\begin{lemma}\label{choc} The space $ML^p(w)$ is order-continuous iff $w\not\in L^1({\mathbb R}^n)$.
\end{lemma}

The proof is based on the following simple property of the maximal operator which can be found in \cite[p. 222]{S93}.

\begin{prop}\label{prmo} Suppose that $w\not\in L^1({\mathbb R}^n)$ and $f\in ML^p(w)$. Then, for all $x\in {\mathbb R}^n$,
$$\lim_{R\to \infty}M(f\chi_{{\mathbb R}^n\setminus B_R})(x)=0,$$
where $B_R$ denotes the ball of radius $R$ centered at the origin.
\end{prop}

\begin{proof}[Proof of Lemma \ref{choc}]
Suppose that $ML^p(w)$ is order-continuous. Then, assuming that $w\in L^1({\mathbb R}^n)$, we arrive at a contradiction by taking the sequence
$f_j:=\chi_{\{|x|\ge j\}}, j\in {\mathbb N}$. Indeed, $f_j\in ML^p(w)$ and $f_j\downarrow 0$ everywhere. However, $Mf_j\equiv 1$
and therefore $\|f_j\|_{ML^p(w)}=\|w\|_{L^1}^{1/p}\ndownarrow 0$.

Suppose now that $w\not\in L^1({\mathbb R}^n)$. Take an arbitrary sequence $\{f_j\}$ in $ML^p(w)$ such that $f_j\downarrow 0$ almost everywhere.
Let $\e>0$. By Proposition~\ref{prmo} (applied to $f_1$) and by the dominated convergence theorem, there exists $R>0$ such that
$$\|f_j\|_{ML^p(w)}\le \|f_j\chi_{B_R}\|_{ML^p(w)}+\e,\quad j\in {\mathbb N}.$$

Consider the sequence $\{M(f_j\chi_{B_R})\}$. By the weak type $(1,1)$ of $M$, this sequence converges to zero in measure, and hence there exists a subsequence
which converges to zero almost everywhere. Since the sequence itself is monotonic decreasing, it also converges to zero almost everywhere. It remains to apply
the dominated convergence theorem, and we obtain that there exists $N\in {\mathbb N}$ such that for all $j>N$,
$$\|f_j\chi_{B_R}\|_{ML^p(w)}<\e,$$
which, combined with the previous estimate, shows that $\|f_j\|_{ML^p(w)}\downarrow~0$. This completes the proof.
\end{proof}

We are now ready to state the main result of this section.

\begin{theorem}\label{msr} Let $p>1$. Next, let $T$ be a linear non-degenerate operator in the sense of Definition \ref{nd}, and assume additionally that $T$ is bounded on $L^p$.
If $T$ is bounded on $ML^p(w)$, then $w\in A_p$.
\end{theorem}

\begin{remark}\label{fol} If $T$ is a non-degenerate (in the sense of Stein) singular integral operator, then, by Remark \ref{std}, it is non-degenerate in the sense of Definition \ref{nd}.
Moreover, it is bounded on $L^p$. Hence, Theorem \ref{msr} contains Theorem \ref{scf} as a particular case.
\end{remark}

\begin{proof}[Proof of Theorem \ref{msr}]
For $\e\in (0,\e_0)$, where $\e_0$ is any fixed small number, define $w_{\e}:=w+\e$.
Then, using that $M$ and $T$ are bounded on $L^p$, we obtain that $T$ is bounded on $ML^p(w_{\e})$ with the $ML^p(w_{\e})$-norm independent of $\e$.
Since $w_{\e}\not\in L^1$, applying Lemma \ref{choc} and Theorem \ref{N} (along with Remark \ref{ndep}), we obtain that $M$ is bounded on $ML^p(w_{\e})$
and on $(ML^p(w_{\e}))'$ with the corresponding norms independent of $\e$. Therefore, by Corollary \ref{fscor} (along with Remark \ref{depc}), $w_{\e}\in A_p$
with the $A_{p}$ constant $[w_{\e}]_{A_p}$ independent of $\e$. Hence, by the monotone convergence theorem, $w\in A_p$.
\end{proof}

We conclude by making a couple of remarks related to Theorem \ref{msr}.

\begin{remark}\label{rpv}
For some singular integral operators Theorem \ref{msr} can be proved differently, without introducing $w_{\e}$ and the use of Remarks \ref{ndep}, \ref{depc}.
For example, for the Hilbert transform $H$, the boundedness of $H$ on $ML^p(w)$
implies the Coifman--Fefferman inequality for $H$, which in turn implies that $w\in C_p$ (this was shown in \cite{M80}). But any $C_p$ weight is not integrable (see, e.g., \cite{CLRT21}), and hence it remains
to apply Theorem~\ref{N} along with Corollary \ref{fscor}.

However, for more general singular integrals such an approach leads to additional technicalities. For example, the $C_p$ condition was deduced in \cite{S83} assuming that the Coifman--Fefferman inequality
holds for {\it all} Riesz transforms $R_j, j=1,\dots,n$. On the other hand, in Theorem \ref{msr}, $T$ can be any {\it one} of the Riesz transforms.
\end{remark}

\begin{remark}\label{r1} A non-degeneracy assumption on $T$ is used in the proof of Theorem \ref{N} through the following property: if $T$ is non-degenerate in the sense of Definition \ref{nd} and weak-$X$ bounded, then
\begin{equation}\label{am}
\sup_{Q}\frac{\|\chi_Q\|_{X}\|\chi_Q\|_{X'}}{|Q|}<\infty,
\end{equation}
where weak-$X$ boundedness means that
$$\sup_{\a>0}\a\|\chi_{\{|Tf|>\a\}}\|_{X}\le C\|f\|_X.$$

At this point, observe that the notion of a non-degenerate operator can be defined in a slightly more general way, which seems more flexible and applicable to a wider class of operators
(even though we do not give concrete examples).

\begin{definition}\label{annd}
We say that $T$ is non-degenerate if there exists a constant $C>0$ such that for every cube $Q$ and any non-negative
locally integrable $f$, one can find a cube $Q'$ such that
\begin{equation}\label{cond1}
\frac{1}{|Q|}\int_Qf\le C|T(f\chi_Q)(x)|\quad\text{for all}\,\,x\in Q'
\end{equation}
and
\begin{equation}\label{cond2}
1\le C|T(\chi_{Q'})(x)|\quad\text{for all}\,\,x\in Q.
\end{equation}
\end{definition}

It is easy to see that if $T$ is non-degenerate in the sense of Definition~\ref{nd}, then it is also non-degenerate in the sense of Definition \ref{annd}.

Next, assuming Definition \ref{annd} and weak-$X$ boundedness, we easily obtain (\ref{am}). Indeed, it follows from (\ref{cond1}) that
$$\Big(\frac{1}{|Q|}\int_Qf\Big)\|\chi_{Q'}\|_{X}\le C\|T\|_{X\to X_{weak}}\|f\chi_Q\|_{X},$$
while (\ref{cond2}) implies
$$\|\chi_{Q}\|_{X}\le C\|T\|_{X\to X_{weak}}\|\chi_{Q'}\|_{X}.$$
Combining both estimates yields
$$\Big(\frac{1}{|Q|}\int_Qf\Big)\|\chi_{Q}\|_{X}\le C^2\|T\|_{X\to X_{weak}}^2\|f\chi_Q\|_{X},$$
which is equivalent to (\ref{am}).

Therefore, Theorems \ref{N} and \ref{msr} can be formulated assuming that $T$ is non-degenerate in the sense of Definition \ref{annd}.
\end{remark}

\end{document}